\documentclass[12pt]{article}
\usepackage{amsmath} 
\usepackage{amssymb} 
\usepackage{dsfont} 
\usepackage{lmodern} 
\usepackage{mathtools} 
\usepackage{xspace} 
\usepackage{amssymb, amsthm}
\usepackage{fancybox} 
\usepackage{graphicx} 
\usepackage{thmtools} 
\include{formatAndDefs} 
\usepackage{tikz} 
\usepackage{frcursive} 
\usepackage{mathrsfs}
\usepackage{booktabs}

\newcommand{\X}{\mathsf{X}}

\newcommand{\A}{\mathcal{A}}
\newcommand{\Real}{\mathbb{R}}
\newcommand{\N}{\mathbb{N}}
\newcommand{\C}{\mathcal{C}}
\newcommand{\D}{\mathcal{D}}
\newcommand{\U}{\mathcal{U}}
\newcommand{\Linear}{\mathcal{L}}
\newcommand{\Partie}{\mathcal{P}}
\newcommand{\T}{\mathcal{T}}
\newcommand{\F}{\mathcal{F}}
\newcommand{\G}{\mathcal{G}}
\newcommand{\M}{\mathcal{M}}
\newcommand{\Hilbert}{\mathsf{H}}
\newcommand{\E}{\mathsf{E}}

\renewcommand{\qedsymbol}{$\bullet$}

\title{Topological method for controllability}
\author{Moustapha Dieye \footnote{ E-mail address : moustapha@aims.edu.gh} }\date{\today}

\begin{document}
\theoremstyle{plain}
\newtheorem{theorem}{Theorem}[section]
\newtheorem{corollary}[theorem]{Corollary}
\newtheorem{lemma}[theorem]{Lemma}
\newtheorem{definition}{Definition}[section]
\newtheorem{remark}{Remark}[section]
\newtheorem{example}{Example}[section]
\newtheorem{proposition}{Proposition}
\renewenvironment{proof}{$\mathbf{Proof}$.}{\begin{flushright}\qedsymbol\end{flushright}}
\maketitle
{\footnotesize
\begin{flushleft}
\hspace{1cm}\thanks{African Institute for Mathematical Sciences, University of Legon, Ghana}\\
\end{flushleft}}
\abstract{In this work, we found a non trivial topology to achieve the controllability for linear and nonlinear system in finite or infinite time horizon. We give several examples illustrating this topologizing method for the controllability results. We obtain by the way the  controllability for the one dimensional Schr\"odinger. We also apply this method to achieve the both controllability of Korteweg-de Vries and Saint-Venant equations.\\
 \textbf{Keywords} : Topologizing set, controllability, approximate controllability, linear, semilinear and nonlinear systems}

\section{The controllability problem}
The concept of controllability plays a very important role in control theory and it's applications. Accordingly a dynamical system is controllable if, with a suitable choice of controls (inputs), it can be driven from any initial state to any desired final state within finite time (or some real applications). It has been extensively studied in the sixties, through eighties, by Fattorini, 
Triggiani, Russell and others, see \cite{Russell, F1, F2, F3, F4, A1, T1} and the references therein. In 1978, Russell made a rather complete survey of the most relevant results that were available in the literature at that time. In \cite{Russell}, the author described a number of different tools that were developed to address controllability problems, often inspired and related to other subjects concerning partial differential equations: multipliers, moment problems, nonharmonic Fourier series, etc.\par

Consider the following system 
\begin{equation}\label{Shor}
\begin{cases}
i\varphi_t + \varphi_{xx} + p(t)x \varphi = 0, \quad 0 < x < 1, 0 < t < T, \\
\varphi(0, t) = \varphi(1, t) = 0, \quad 0 < t < T, \\
\varphi(x, 0) = \varphi^0 (x), 0 < x < 1, \\
\end{cases}
\end{equation}
wehre $\varphi = \varphi(x, t)$ is the state and $p = p(t)$ is the control. Although $\varphi$ is 
complex-valued, $p(t)$ is real for all $t$. The control $p$ is interpreted as the intensity of an applied electrical field and $x$ is the (prescribed) direction of the laser. The state $\varphi = \varphi(x, t)$ is the wave function of the molecular system. One can regard it as a function that furnishes information on the location of an elementary particle : for arbitrary $a$ and $b$ with $0 
\leq a < b \leq 1$, the quantity 
$$P(a, b; t) = \int_ a^b |\varphi(x, t)|^2 dx$$ 
can be viewed as the probability that the particle is located in $(a, b)$ at time $t$. 
The controllability problem for the system \eqref{Shor} is to find the set of attainable states $\varphi(\cdot, T )$ at a final time $T$ as $p$ runs over the whole space $L^2 (0, T )$. 
\par 
The control problems arising in this context are bilinear. This adds fundamental difficulties from a mathematical viewpoint and makes these problems extremely challenging. Indeed, we find here genuine nonlinear problems for which, apparently, the existing linear theory is insufficient to provide an answer in a first approach.\par 
In \cite{zuazua}, the author said that the above controllability problem is a bilinear control problem, since the unique nonlinearity in the model is the term $p(t)x \varphi$, which is essentially the product of the control and the state. Although the nonlinearity might seem simple, this control problem becomes rather complex and out of the scope of the existing methods in the literature. \par 
Motivated by the above situations, we try here to find an alternative method to deal with this controllability problem. Let $\X$ be the outpout space and $\mathsf{U}$ the space of inputs. We consider the following system of the form 
\begin{equation}\label{A}
dx(t)/dt = f(x(t), u(t)), \quad x(0)=x_0, \quad t\geq0 
\end{equation}
in the Banach space $\X$ with controls $u$ being suitable functions taking values $u(t)$ in another Banach space $\mathsf{U}$. $f:\X \times\mathsf{U}\rightarrow\X$, the system function, is a given function. The functions $u$ belong to a class $\U_{ad}$ of admissible controls i.e there are elements of some class (set) noted here $\U_{ad}$ such that for any $u\in\U_{ad}$ the system \eqref{A} has a (mild) solution. 

 For any $t\in[0, T]$, the attainable set is define by
\begin{equation}\label{attainable}
\A(t):=\left\{x^u (t), x^u\text{ a solution of }\eqref{A}\text{ corresponding to } u\in \U_{ad}\right\}.
\end{equation}
Naturally, the attainable set depends on the definition of solution (or mild solution) of the dynamic processes of our interest. 
The system \eqref{A} is said to be controllable (exactly controllable) over the time interval $[0, t]$ if $\A(t)$ is dense in $\X \left( \A(t) = \X 
\right))$. The system is said to be controllable (exactly controllable) if $\displaystyle\cup_{t\geq0}\A(t)$ is dense in $\X \left( \displaystyle\cup_{t\geq0} \A(t) = \X \right)$.\par The approximate controllability depends on the strong topology of the Banach space $\X$. 
For example, let us consider the following  controlled system 
\begin{equation}\label{Trivial}
\begin{cases}
x_1(t)=c,\\
\dfrac{dx_2(t)}{dt}=u,\quad u\in\Real
\end{cases},t\in[0,T],T<\infty,
\end{equation}
where  $c$ is a constant and $u$ is the control. The corresponding attainable set is given by $\A(T)=\left\{c\right\}\times\Real$.
It is clear that this system \eqref{Trivial} cannot be approximately controllable according the usual topology on $\Real\times\Real$ ; it is when we put the trivial topology, namely the empty set and full state space. The trivial topology is too weak to be used.  
\par
When the exact controllability is impossible, it is clear that the remaining controllability problem depend on the involvoing topology on $\X$.
 If the attainable set is empty then there is nothing to be control.  In fact, the controllability problem is a verification of density of the attainable set. Therefore, one can found a criteria for controllability. The considered system and the controls cannot be chosen arbitrary. They depend on the the real nature of the studied problem. 
 
 Therefore, the problem is to found a topology $\F$ (non trivial) on $\X$ where the attainable set is dense. We propose this approach for the study of controllability of any given dynamical system. The trivial topology always guarantees the approximate controllability of any controlled system provided that there exists at least a trajectory starting from a given initial data. In the present work, we found a non trivial topology that achieve approximate controllability. 
\par
The rest of the manuscript is organized as follows. Section \ref{S1} is a presentation of the topologizing method used to construct a topology where a given subset is dense. Several controllability results are obtained, in the last Section \ref{S2}, as examples for illustrations.

\section{A topologizing method}\label{S1}
In this section, we present a concept of topologizing set to develop this method for controllability For a given set $X$, the set $\Partie(X)$ denotes all the subsets of $X$. We recall for consistence that the difference $A-B$ of to sets is $\left\{x\in A: x\notin B\right\}$. If $A\subset X$, the complement $A^c$ of $A$ with respect to $X$ is $X-A$. From now on, we denote by $\emptyset$ the empty set and we recall the following topologizing set result 
\begin{lemma}\cite[p.73]{Dugundji}\label{Mu}
Let $X$ be a set, and $\gamma:\Partie(X)\rightarrow\Partie(X)$ a map with the properties :
\begin{itemize}
\item[$(\mathsf{i})$]\hspace{3cm}$\gamma(\emptyset)=\emptyset.$
\item[$(\mathsf{ii})$]\hspace{3cm}$A\subset\gamma(A)$ for each $A\in\Partie(X).$ 
\item[$(\mathsf{iii})$]\hspace{3cm}$\gamma\circ\gamma(A)=\gamma(A)$ for each $A\in\Partie(X).$
\item[$(\mathsf{iv})$]\hspace{3cm} $\gamma(A\cup B)=\gamma(A)\cup \gamma(B)$ for each $A, B\in\Partie(X).$
\end{itemize}
Then the famaly $\D=\left\{\gamma(A)^c:A\in\Partie(X)\right\}$ is a topology on $X$ and $\overline{A}=\gamma(A)$ is the closure of the set $A\in\Partie(X)$.
\end{lemma}
Let $(X, \G)$ be a topological space and $F$ a strict non empty subset of $X$. We define the following map from $\Partie(X)$ to itself by
\begin{equation*}
\mu(A)=
\begin{cases}
\begin{array}{ll}
\emptyset&\text{ if }A=\emptyset\\
A\cup F^c&\text{ else} 
\end{array}
\end{cases}
\end{equation*}
We give the following Lemma
\begin{lemma}\label{CT}
The set $\F=\left\{\mu(A)^c:A\in\Partie(X)\right\}$ is a topology on $X$ where $F$ is dense.
\end{lemma}
\begin{proof}
The proof is a simple verification of the assumptions of the above Lemma for the map $\mu$. In fact, it is clear that the items $(\mathsf{i})$ and $(\mathsf{ii})$ are satisfied.\par
Let $A\in\Partie(X)$. We have the following
\begin{eqnarray*}
\mu\circ\mu(A)&=&\mu\left(
\begin{cases}
\begin{array}{ll}
\emptyset&\text{ if }A=\emptyset\\
A\cup F^c&\text{ else} 
\end{array}
\end{cases}\right)\\
&=&\mu(A)
\end{eqnarray*}
Therefore $(\mathsf{iii})$ is verified.
\par 
Let $A, B\in\Partie(X)$.
If $A$ or $B$ is empty then $\mu(A\cup B)=\mu(A)\cup\mu(B)$. So we suppose that $A$ and $B$ are not empty. Then, 
$$\mu(A\cup B)=A\cup B\cup F^c=(A\cup F^c)\cup( B\cup F^c)=\mu(A)\cup\mu(B).$$
By Lemma \ref{Mu}, $\F$ is a topology on $X$ and since $\overline{F}=\mu(F)=X$, the density of $F$ in $X$ follows. 
Hence $(X, \F)$ is a topology where $F$ is dense. 
\end{proof}
 In the topological space $(X,\G)$ for $Y\subset X$ we recall that the induced topology $\G_Y$ on $Y$ is $\left\{Y\cap U;U\in\G\right\}.$ We also recall that given two topologies $\G_1$ and $\G_2$ on $X$, $\G_2$ is larger than $\G_1$ whenever $\G_1\subset \G_2$. It is important to note that 
 \begin{eqnarray}
\F&=&\left\{\mu(A)^c,A\in\Partie(X)\right\}\nonumber\\
&\equiv&\left\{A^c\cap F,A\in\Partie(X),A\neq\emptyset\right\}\bigcup\left\{X\right\}.\label{Trace}
 \end{eqnarray}
 Therefore, we see that the closure of $F$ is the set $X$. Further, we give the following 
\begin{corollary}
The induced topology $\F_F$ is larger than $\G_F$ and $\G_F\subset\F$.
\end{corollary}
\begin{remark}\label{Rq}
Since $F$ is a strict part of $X$, one see clearly that $\left\{\emptyset,X\right\}\subsetneq\F\subsetneq\Partie(X)$.\\
If $F^c\in\G$ that $\mu(\G)^c:=\left\{\mu(\theta)^c:\theta\in\G\right\}\subsetneq\G$. In fact, $\mu(\G)^c$ is not stable under arbitrairy union ; it is for finite union. Moreover, the topology $(X, \F)$ is separable if $F$ is countable. 
\end{remark}
Nagata-Smirnov theorem gives the necessary and sufficient conditions of metrization of a given topological space. Among others, a condition of separation  stronger that Haussdorf is required, see \cite{Dugundji}. Thus the following result states that the above constructed topology is not metrizable. 
\begin{theorem} The topology $\F$, given by \ref{Trace}, is not Hausdorff.
\end{theorem}
\begin{proof}
Suppose that  $\F$ is Hausdorff.  Let $x,y\in \X$ such that  $x\neq y$. In particular, we take  $x\notin F$ (this is possible, since $F$ is a propre subset of $X$). There exist $U,V\in\F$ such that $U\cap V=\emptyset$, $x\in U$ and $y\in V$, that is,  from \eqref{Trace}, the existence of $A,B\in\Partie(X)$, non empty sets, such that $U=A^c \cap F$ and $V=B^c\cap F$.  Thus $x\in F$. This is impossible and so 
we have obtained a contradiction, proving our theorem.
\end{proof}
One notes that $\F$ separated the set $F$ in the sense that each two distinct points of $F$ have nointersecting neighbourhoods in $\F$.
But,  for consistency, we add  the following subsection 
\subsection{Nets}
In a metric space, a point belongs to the closure of a given set if and only if it is the limit of some sequence of points belonging to that set. The convergence of the sequence $(a_n)_{n\in \N}$ to the point $x$ is defined by the requirement that for any neighbourhood $U$ of $x$ there is some $N \in \N$ such that an belongs to $U$ whenever $n \geq N$. As the definition of convergence of the sequence $(a_n )_{n \in \N}$ in an arbitrary topological space, the resulting theory is not as 
straightforward as one might suspect, as is illustrated by the following example.
\begin{example} 
Let $X$ be the real interval $[0, 1]$ and let $\T$ be the co-countable topology on $\X$ ; that is, $\T$ consists of $X$ and $\emptyset$ together with all those subsets $G$ of $X$ whose complement $X\setminus G$ is a countable set. Let $A = [0, 1)$, we have that $\overline{A} = [0, 1]$, since $[0, 1]$ is the only closed set containing $A$. $1\notin A, $ it must be a limit point of $A$. No sequence in $A$ can converge to the limit point 1.
\end{example} 
 We have exhibited a topological space with a subset possessing a limit point which is not the limit of any sequence from the 
subset. \par A sequence in a set $X$ is a mapping from $\N$ into $X$, where it is implicitly understood that $\N$ is directed via its usual order structure. We generalize this to general directed sets. 
\begin{definition}
A directed set is a partially ordered set $J$, with partial order 
$<$, say, such that for any pair of elements $\alpha, \beta \in J$ there is some $\delta\in J$ such that $\alpha <\delta$
and $\beta< \delta$. 
\end{definition}
\begin{definition}
 A net in a topological space $(X, \T)$ is a mapping from a directed 
set $J$ into $X$; denoted $(x _{\alpha} )_{\alpha \in J}$. 
\end{definition}
If $J=\N$ with its usual ordering, then we recover the notion of a sequence. 
\begin{definition}
 Let $(x _{\alpha} )_{\alpha \in J}$ be a net in a topological space $(X, \T)$.
\begin{itemize} 
\item[-] The net $(x _{\alpha} )_{\alpha \in J}$ is eventually in the set $A$ if there is $\beta \in I$ such that $x_{\alpha} \in A$ whenever 
$\alpha > \beta$. 
\item[-] The net $ (x _{\alpha} )_{\alpha\in J}$ converges to the point $x\in X$ if, for any neighbourhood 
$U$ of $x, (x _{\alpha} )_{\alpha\in J}$ is eventually in $U$.
\end{itemize} $x$ is called a limit of $(x_{\alpha} )_{\alpha\in J}$, and we write 
$x _{\alpha} \rightarrow x$ (along $I$). 
\end{definition}
We can characterize the closure of a set $A$ as that set consisting of all the limits 
of nets in $A$.
\begin{theorem}\cite{Dugundji}
 Let $A$ be a subset of a topological space $(X, \T)$. Then $x \in \overline{A}$ if and only if there is a net $(a_{ \alpha} )_{\alpha \in J}$ with $a_{\alpha} \in A$ such that $a_{\alpha} \rightarrow x$. 
\end{theorem}
Accordingly, we get the following result
\begin{lemma} Suppose that $F$ is closed in $\G$. Let $ (x _{\alpha} )_{\alpha\in J}$, a net, converges to the point $x\in X$ then for every $\theta\in\G_F$ containing $x$ there exist $\beta\in J$ such that $x_{\alpha}\in\theta$ for all $\alpha>\beta$.
\end{lemma}
\section{Some examples}\label{S2}
In this section, we use the topologizing method to obtain several controllability results. In fact, applying Lemma \ref{CT}, one has the following result
\begin{corollary}\label{Control}
If the attainable set is not empty then the system \eqref{A} is approximately controllable on $[0, T]$ i.e. there exists is a topology, $\F$, on $\X$ where $\A(T)$ is dense.
\end{corollary}
Therefore, we have the following
\begin{proposition}
 The system \eqref{Shor} is controllable even in the particular case $\varphi^0=0$ which was left open in \cite{zuazua}.
\end{proposition}
For an overview on the state of the art of the control of systems governed by the Schr\"odinger equation, we refer to the survey article \cite{243} and the references therein. 
\subsection{Controllability of semilinear systems} 
We consider the semilinear system given by 
\begin{equation}\label{Sem}
 dx(t) = \mathsf{A}x(t)dt + \mathsf{B}(t)u(dt) + f (t, x(t))dt, x(0) = \xi,\quad t \in I=[0,T]
\end{equation}
where $\mathsf{A}$ is the infinitesimal generator of a $C_0$ semigroup of operatrors $(S(t))_{ t \geq 0}$, on the Hilbert space $\Hilbert$ and $u \in \M_c (I, \E)$ is a control, vector measure with values in the Hilbert space $\E$, and $\mathsf{B}$ is a uniformly measurable bounded operator valued function with values in $\Linear(\E, \Hilbert )$ and $f$ is a measurable map from $I\times \Hilbert$ to $\Hilbert$ continuous in the second argument.\par
In \cite{Ahmedvector}, the author shows that if, $\mathsf{A}$ generates a compact semigroup $S (t), t > 0$, in $\Hilbert$ when the dominating measure $\mu$ has no atom at $\{T \}$, 
or if, $A$ generates an analytic semigroup, the system \eqref{Sem} is not exactly controllable. \par
The controllability (approximate) of the linear part of system \eqref{Sem} is characterized by the famous observability inequality ( Carleman's inequalities) which is sometime hardly established. 
Using the previous result of the above Section, we guarantee the existence of a topology where the semilinear system \eqref{Sem} is controllable. 
\begin{proposition}
Assume that the system \eqref{Sem} is not approximately controllable according to Borel topology (for example it linear part is not controllable if $\mathsf{A}$ is not positive). Nevertheless, there exists a topology on $\X$ such that the system \eqref{Sem} is approximately controllable.
\end{proposition}
\subsection{The Korteweg-de Vries control system}
Prior pioneer work on the controllability of the Korteweg-de Vries equation 
(with periodic boundary conditions and internal controls) was due to D. Russell and 
B.-Y. Zhang (1996). Among other method, they use the linearized control system to establish the a local controllability result. Here we propuse a different apporoach, we recall the Korteweg-de Vries control system
\begin{equation}\label{KV}
\begin{cases}
y_t + y_x + y_{xxx} + yy_x = 0, \quad t \in [0, T ], x \in [0, L], \\
y(t, 0) = y(t, L) = 0, \quad t \in [0, T ], \\
 y_x(t, L) = u(t), \quad t \in [0, T ], 
\end{cases}
\end{equation}
where, at time $t \in [0, T ]$, the control is $u \in \Real$ and the state is $y(t, \cdot) \in L^2 (0, L)$. \par
Setting $\X=L^2 (0, L)$, $y(t)=y(t, \cdot)$ and $f(y(t), u(t))=-(\partial_x+\partial_{xxx})y(t)+y(t) u(t)$
Corollary \ref{Control} states that
\begin{proposition}
The system \eqref{KV} is approximately controllable in the sens that there exists a topology on $L^2 (0, L)$ where the attainable set $\A(T)$ corresponding to the controlled Korteweg-de Vries equation \eqref{KV} is dense in $L^2 (0, L)$. Therefore, we have for any $g\in L^2(0, L)$ the existence of a net $(g_{ \alpha} )_{\alpha \in J}$ with $g_{\alpha} \in \A(T)$ such that $g_{\alpha} \rightarrow x$.
\end{proposition}
\subsection{Water-tank control}
Saint-Venant equations are derived from Navier-Stokes Equations for shallow water flow conditions. The Navier-Stokes Equations are a general model which can be used to model water flows in many applications. A general flood wave for $1-$demensional situation can be described by the Saint-Venant equations. The shallow water equations can been dericated from the Euler equations for the perfect irrotational and incompressible fluids \cite{[4]}.
\par The shallow water equations describe the motion of a perfect fluid under gravity $g$ :
\begin{equation}\label{SV}
\begin{cases}
\begin{array}{rcll}
H_t + (H v)_x &=& 0,& \quad t \in [0, T ], x \in [0, L], \\
v_t + \left( gH +\dfrac{v}{2} \right)_x &=&-u (t),&\quad t \in [0, T ], x \in [0, L],\\
v(t, 0) &=& v(t, L) = 0,&\quad t \in [0, T ], \\ 
s'(t) &=& u (t),& \quad t \in [0, T ],\\ 
D'(t) &=& s (t),&\quad t \in [0, T ],
\end{array} 
\end{cases} 
\end{equation}
where $x \in [0, L]$ is the spatial coordinate attached to the tank, $L$ is of length ot the tank containing the fluid, $t \in [0, T ] $is the 
time coordinate, $T > 0$, $H (t, x)$ denotes the height of the liquid, $v (t, x)$ denotes the horizontal velocity of the fluid in the referential attached to the tank, $u (t)$ is the horizontal acceleration of the tank in the absolute referential, $D$ is the horizontal displacement of the tank, 
$D'$ is the first derivative of $D$ with respect to the time $t$, $s$ is the horizontal velocity of the tank. \par
 \par To build the state space, we suppose that 
\begin{eqnarray}
&&\frac{d}{dt}\int_0^LH(t,x)dx = 0, \\ 
&&H_x(t, 0) = H_x(t, L) (=-u(t)/g). 
\end{eqnarray} 
The state space (denoted $\X$ ) is the set of 
$ K= (H, v, s, D) \in \C^1 ([0, L]) \times \C^1 ([0, L])\times \Real \times \Real$ 
satisfying 
\begin{equation}
v(0) = v(L) = 0, H_x (0) = H_x(L), \int_0^LH(x)dx = LHe. 
\end{equation}
For the controllability problem, the classical method of locally controllability of the linearized system.The linearized shallow water equations around some equilibrium are uncontrollable, even locally. We refer to \cite{[7]}. 
Once again, one uses the above topologizing method to state that 
\begin{proposition}
For any $T > 0$ the water-tank control system \eqref{SV} is approximately controllable in time $T$ i.e. there exists a topology $\F$ on $\X$ such that the coresponding attainable set is dense in $\F$ 
\end{proposition} 

\section*{Conclusion} Given a subset of a set, we found an easy way to topologize (which is non trivial topology) the bigger set such that the smaller set is dense. This possibility gives rise to the following question : can we topologize in order to get a non trivial metrizable topology where our smaller non empty set is dense ?

\end{document}